\documentclass[12pt,reqno]{article}

\usepackage{color}
\usepackage{fullpage}
\usepackage{float}

\definecolor{webgreen}{rgb}{0,.5,0}
\definecolor{webbrown}{rgb}{.6,0,0}

\usepackage[colorlinks=true,
linkcolor=webgreen,
filecolor=webbrown,
citecolor=webgreen]{hyperref}

\usepackage{amsmath,amssymb}
\usepackage{amsthm}
\usepackage{amsfonts}
\usepackage{latexsym}
\usepackage{epsf}

\usepackage{qtree}

\newcommand{\seqnum}[1]{\href{http://oeis.org/#1}{\underline{#1}}}

\begin{document}


\theoremstyle{plain}
\newtheorem{theorem}{Theorem}
\newtheorem{corollary}[theorem]{Corollary}
\newtheorem{lemma}[theorem]{Lemma}
\newtheorem{proposition}[theorem]{Proposition}

\theoremstyle{definition}
\newtheorem{definition}[theorem]{Definition}
\newtheorem{example}[theorem]{Example}
\newtheorem{conjecture}[theorem]{Conjecture}
\newtheorem{problem}[theorem]{Problem}

\theoremstyle{remark}
\newtheorem{remark}[theorem]{Remark}

\newcommand\T{\mathcal{T}}
\newcommand\CC{\mathbb{C}}

\begin{center}
\vskip 1cm{\LARGE\bf The Nonassociativity of the Double Minus  \\
\vskip .1in Operation}
\vskip 1cm
\large
Jia Huang, Madison Mickey, and Jianbai Xu\footnote{
The second and third authors were supported by the Undergraduate Research Fellows program of the University of Nebraska at Kearney.} \\
Department of Mathematics and Statistics \\
University of Nebraska at Kearney \\
Kearney, NE 68849  \\
USA \\
\href{mailto:huangj2@unk.edu}{\tt huangj2@unk.edu} \\
\href{mailto:mickeym2@lopers.unk.edu}{\tt mickeym2@lopers.unk.edu} \\
\href{mailto:xuj2@lopers.unk.edu}{\tt xuj2@lopers.unk.edu} \\
\end{center}

\vskip .2 in

\begin{abstract}
The sequence \seqnum{A000975} in OEIS can be defined by $A_1=1$, $A_{n+1}=2A_n$ if $n$ is odd, and $A_{n+1}=2A_n+1$ if $n$ is even.
This sequence satisfies other recurrence relations, admits some closed formulas, and is known to enumerate several interesting families of objects.
We provide a new interpretation of this sequence using a binary operation defined by $a\ominus b := -a -b$.
We show that the number of distinct results obtained by inserting parentheses in the expression $x_0\ominus x_1\ominus \cdots\ominus x_n$ equals $A_n$, by investigating the leaf depth in binary trees.
Our result can be viewed as a quantitative measurement for the nonassociativity of the binary operation $\ominus$.
\end{abstract}

\section{Introduction}\label{sec:intro}

The sequence \seqnum{A000975} in \emph{The On-Line Encyclopedia of Integer Sequences} (OEIS)~\cite{OEIS} starts with $1,2,5,10,21,42,85,\ldots$.
With its $n$th term denoted by $A_n$, this sequence has the following (equivalent) characterizations.
\begin{itemize}
\item
$A_1=1$, $A_{n+1}=2A_n$ if $n$ is odd, and $A_{n+1}=2A_n+1$ if $n$ is even.
\item
$A_1=1$ and $A_n=2^n-1-A_{n-1}$ for $n\ge2$.
\item
$A_1=1$, $A_2=2$, and $A_n=A_{n-2}+2^{n-1}$ for $n\ge3$.
\item
$A_n$ is the positive integer with an alternating binary representation of length $n$. 
\end{itemize}

There are also some known closed formulas for this sequence:
\begin{equation}\label{eq:An}
A_n= \left\lfloor \frac{2^{n+1}}3 \right\rfloor = \frac{2^{n+2}-3-(-1)^n}{6} 
=\begin{cases} 
\displaystyle \frac{2^{n+1}-1}3, & \text{if $n$ is odd}; \\[15pt]
\displaystyle \frac{2^{n+1}-2}3, & \text{if $n$ is even}. 
\end{cases} 
\end{equation}

By Hinz~\cite{Hinz}, this sequence dates back to 1769, when 
Lichtenberg studied it in connection with the Chinese Rings puzzle (baguenaudier).
This sequence has many interesting interpretations, some of which are listed below.
\begin{itemize}
\item
$A_n$ is the number of moves required to solve the $n$-ring Chinese Rings puzzle~\cite{Hinz}.
\item
$A_n$ is the distance between $0^n$ and $1^n$ in an $n$-bit binary Gray code~\cite[Ch. 1]{HanoiTower}.
\item
$A_n$ is the number of ways to partition $n+2$ people sitting at a circular table into three nonempty groups with no two members of the same group seated next to each other~\cite{CircSet}.
\end{itemize}

The sequence \seqnum{A000975} also has connections to some other sequences in OEIS~\cite{OEIS}, such as \seqnum{A000217}, \seqnum{A048702}, \seqnum{A155051}, \seqnum{A265158}, and so on.
See Stockmeyer~\cite{A000975} for more details.

In this paper we present a new interpretation for the sequence \seqnum{A000975} using a binary operation which is not associative.
This naturally connects to the ubiquitous Catalan numbers and objects counted by Catalan numbers, such as binary trees.

In fact, let $*$ be a binary operation defined on a set $S$ and let $x_0, x_1,\ldots, x_n$ be indeterminates taking values from $S$.
It is well known that the number of ways to insert parentheses in the expression $x_0 * x_1 * \cdots * x_n$ is the \emph{Catalan number} $C_n:=\frac{1}{n+1} \binom{2n}{n}$, which enumerates many families of interesting objects~\cite{EC2}.

A natural question is how many distinct results can be obtained by inserting parentheses in $x_0 * x_1 * \cdots * x_n$ for a given binary operation $*$.
To be more precise, observe that each parenthesization of $x_0 * x_1 * \cdots * x_n$ gives a function from $S^{n+1}\to S$, and if two parenthesizations give the same function then they are said to be \emph{$*$-equivalent}.
Hein and the first author~\cite{CatMod} defined $C_{*,n}$ to be the number of $*$-equivalence classes of parenthesizations of $x_0 * x_1 * \cdots * x_n$.
This provides a quantitative measurement for the nonassociativity of $*$, as we have inequalities $1\le C_{*,n}\le C_n$ where the first equality holds if and only if $*$ is \emph{associative}, i.e., $(a*b)*c=a*(b*c)$ for all $a,b,c\in S$.

Lord~\cite{Lord} introduced the \emph{depth of nonassociativity} of a binary operation and examined this depth for some elementary binary operations.
The nonassociativity measurement $C_{*,n}$ is much finer than the depth of nonassociativity of $*$ since the latter can be written as $\inf \{ n+1: C_{*,n}<C_n\}$.
We have not found any other result on measuring nonassociativity of a binary operation in the literature.

For a family of binary operations defined by $a*b:= \omega a + b$ for $a$ and $b$ in the field $\CC$ of complex numbers, where $\omega := e^{2\pi i/k}$ is a primitive $k$th root of unity, Hein and the first author~\cite{CatMod} determined the number $C_{*,n}$ and studied its connections to various objects counted by the Catalan numbers.

Now we define a binary operation by $a \ominus b := -a-b$ for all $a,b$ in $\CC$ (or any other infinite field).
We call this operation the \emph{double minus operation}.
It is not associative and we show that its nonassociativity measurement $C_{\ominus,n}$ equals $A_n$, the $n$th number in the sequence \seqnum{A000975}.

To prove this result, we first study the correspondence between parenthesizations and binary trees, and develop some lemmas on the depths of leaves in binary trees in Section~\ref{sec:prelim}.
Here the \emph{depth} of a leaf is the number of steps in the unique path from the root to this leaf.

We are interested in the leaf depths in binary trees because the results from parenthesizations of $x_0 \ominus x_1\ominus \cdots \ominus x_n$ can be written in the form
\[ \pm x_0 \pm x_1 \pm\cdots \pm x_n \]
where the sign of $x_i$ is determined by the parity of the depth of the $i$th leaf in the corresponding binary tree.
Combining this observation with the lemmas given in Section~\ref{sec:prelim} we prove our main results in Section~\ref{sec:main}, which are summarized below.

We first establish a closed formula for the number $C_{\ominus,n,r}$ of distinct results obtained from parenthesizations of $x_0\ominus x_1\ominus \cdots \ominus x_n$ with exactly $r$ plus signs.
The result gives a truncated/modified version of the well-known Pascal triangle; see Table~\ref{tab:Cnr}.

The number $C_{\ominus,n,r}$ refines the number $C_{\ominus,n}$ as $C_{\ominus,n} = \sum_{0\le r\le n+1} C_{\ominus,n,r}$.
It also implies that $C_{\ominus,n}$ is determined by a ``skipping sum'' of binomial coefficients and we explicitly calculate it using cubic roots of unity.
This proves the equality $C_{\ominus,n}=A_n$.
Our proof leads to a characterization of all distinct results from parenthesizations of $x_0 \ominus x_1\ominus \cdots \ominus x_n$ using binary sequences, and hence gives another family of objects enumerated by the sequence \seqnum{A000975}.
We also present a simpler proof for the equality $C_{\ominus,n}=A_n$, which was pointed out to us by an anonymous referee, by showing the recurrence relation $C_{\ominus,n} = 2^n-1-C_{\ominus,n-1}$ for $n\ge2$.

Finally, we provide some questions related to our results and proofs for future study in Section~\ref{sec:questions}.

\section{Leaf depths in binary trees}\label{sec:prelim}

In this section we first recall some basic definitions and properties for binary trees and their connections to parenthesizations, and then develop some lemmas on the depths of leaves in binary trees. 

A \emph{(full) binary tree} is a rooted tree in which each node has either no child or two children ordered from left to right.
A node without any child is called a \emph{leaf}. 

Let $t$ be a binary tree. 
A \emph{subtree} of $t$ is a binary tree whose edges and nodes are contained in $t$.
The \emph{left/right subtree} of a node $v$ in $t$ is the maximal subtree of $t$ rooted at the left/right child of $v$.

Let $\T_n$ be the set of all binary trees with $n+1$ leaves.
For each $t\in\T_n$, we label the leaves $0,1,\ldots,n$ from left to right, or more precisely, in the \emph{preorder} of $t$.
The \emph{depth} of leaf $i$ in $t$, denoted by $d_i(t)$, is the length of the unique path from the root of $t$ down to leaf $i$.
The \emph{depth sequence} of $t$ is defined as $d(t) = ( d_0(t),d_1(t),\ldots,d_n(t) )$.

Next, we introduce a binary operation $\wedge$ for binary trees.
If $s\in\T_m$ and $t\in\T_n$ then define $s\wedge t$ to be the binary tree in $\T_{m+n+1}$ whose root has left subtree $s$ and right subtree $t$.

\begin{proposition}\label{prop:wedge}
If $s\in\T_m$ and $t\in\T_n$ then the depth sequence of $s\wedge t$ is 
\[ d(s\wedge t) = ( d_0(s)+1, \ldots,d_m(s)+1, d_0(t)+1,\ldots,d_n(t)+1).\]
\end{proposition}

\begin{proof}
The result follows immediately from the construction of $s\wedge t$.
\end{proof}

\begin{example}
Figure~\ref{fig:wedge} illustrates binary trees $s\in\T_3$, $t\in\T_2$, and $s\wedge t\in\T_6$.
One sees that $d(s) = (2,3,3,1)$, $d(t) = (2,2,1)$, and $d(s\wedge t) = (3,4,4,2,3,3,2)$.
\begin{figure}[H]
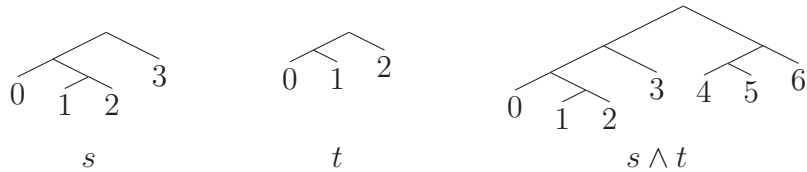

\[ \begin{array}{ccccc}
\Tree [.  [. 0 [. 1 2 ] ] 3 ] & \qquad &  \Tree [. [. 0 1 ] 2 ] & \qquad & \raisebox{10pt}{\Tree [. [.  [. 0 [. 1 2 ] ] 3 ] [. [. 4 5 ] 6 ] ]} \\
s & \qquad & t & \qquad & s\wedge t
\end{array} \]
\caption{A binary operation on binary trees}\label{fig:wedge}
\end{figure}
\end{example}

Let $*$ be a binary operation on a set $S$ and let $x_0,x_1,\ldots,x_n$ be $S$-valued indeterminates.
There are different ways to insert parentheses in the expression $x_0*x_1*\cdots*x_n$.
These parenthesizations of $x_0*x_1*\cdots*x_n$ are enumerated by the {Catalan number} $C_n$.
Moreover, the parenthesizations of $x_0*x_1*\cdots*x_n$ are naturally in bijection with binary trees in $\T_n$, since these binary trees are precisely the results from parenthesizing $t_0\wedge t_1 \wedge \cdots \wedge t_n$, where $t_0,t_1,\ldots,t_n$ all consist of a single node.

\begin{example}\label{example:correspondence}
We list all binary trees in $\T_3$ and their corresponding parenthesizations in Figure~\ref{fig:T2P}.
These binary trees have depth sequences $(3,3,2,1)$, $(2,2,2,2)$, $(2,3,3,1)$, $(1,3,3,2)$, and $(1,2,3,3)$, respectively.
\begin{figure}[H]
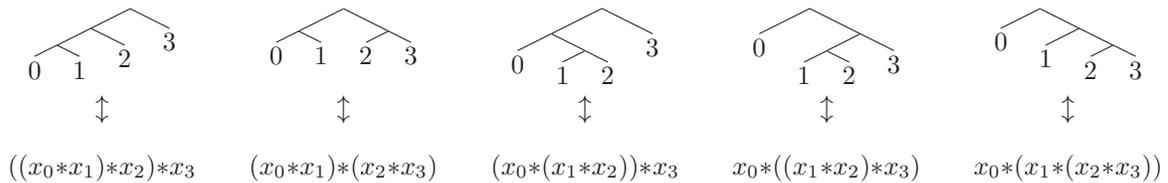

{\footnotesize \[ \begin{array}{ccccc}
\Tree [.  [. [. 0 1 ] 2 ] 3 ] &
\Tree [.  [. 0 1 ] [. 2 3 ] ] &
\Tree [.  [. 0 [. 1 2 ] ] 3 ] &
\Tree [. 0 [. [. 1 2 ] 3 ] ] &
\Tree [. 0  [. 1 [. 2 3 ] ] ] \\[10pt]
\updownarrow & \updownarrow & \updownarrow & \updownarrow & \updownarrow \\[10pt]
\rule{5pt}{0pt}((x_0 {*} x_1) {*} x_2) {*} x_3\rule{5pt}{0pt} &
\rule{5pt}{0pt}(x_0 {*} x_1) {*} (x_2 {*} x_3)\rule{5pt}{0pt} &
\rule{5pt}{0pt}(x_0 {*} (x_1{*} x_2)) {*} x_3 \rule{5pt}{0pt} &
\rule{5pt}{0pt}x_0 {*} ((x_1 {*} x_2) {*} x_3)\rule{5pt}{0pt} &
\rule{5pt}{0pt}x_0 {*} (x_1 {*} (x_2 {*} x_3))\rule{5pt}{0pt}
\end{array} \] }
\caption{Correspondence between binary trees and parenthesizations}
\label{fig:T2P}
\end{figure}
\end{example}

In the remainder of this section we give some lemmas on leaf depths in binary trees.

\begin{lemma}\label{lem:DNE}
For all $n\ge0$ and $0\le r\le n+1$, there exists a binary tree $t\in\T_n$ with exactly $r$ leaves of even depth only if $n+r\equiv1\pmod 3$.
\end{lemma}

\begin{proof}
We induct on $n$.
The case $n=0$ is trivial.
Assume $n\ge1$, $0\le r\le n+1$, and there exists a binary tree $t\in\T_n$ with exactly $r$ leaves of even depth. 
Let $t_1$ and $t_2$ be the left and right subtrees of the root of $t$.
For $i=1,2$, let $n_i+1$ be the number of leaves in $t_i$ and let $r_i$ be the number of leaves of even depth in $t_i$.
Since $t=t_1\wedge t_2$, we have 
\[ n_1+1+n_2+1=n+1 \qquad\text{and}\qquad r_1 + r_2 = n+1-r\]
where the second equation follows from Proposition~\ref{prop:wedge}.
We also have 
\[ n_1+r_1 \equiv 1\pmod3 \qquad\text{and}\qquad n_2+r_2 \equiv 1\pmod3 \]
by the induction hypothesis.
Therefore
\[ n+r = 2(n_1+n_2)+3-(r_1+r_2) \equiv 2(n_1+n_2+r_1+r_2) \equiv 1 \pmod 3. \]
This completes the proof.
\end{proof}

\begin{lemma}\label{lem:alternating}
For all $n\ge1$ and $t\in\T_n$, we have $d_j(t) = d_{j+1}(t)$ for some $j\in\{0,1,\ldots,n-1\}$.
\end{lemma}
\begin{proof}
Consider the leftmost leaf $j$ with the largest depth among all leaves of $t$. 
The sibling of $j$, i.e., the other child of the parent of $j$, must be a leaf (otherwise the depth of $j$ is not the largest) and to the right of $j$ (otherwise $j$ is not the leftmost leaf with the largest depth).
It follows that the sibling of $j$ is the leaf $j+1$ and $d_j(t) = d_{j+1}(t)$.
\end{proof}

Let $n\ge0$. 
We say a binary sequence $(d_0,d_1,\ldots,d_n)\in\{0,1\}^{n+1}$ is \emph{non-alternating} if $n\ge1$ implies $d_j = d_{j+1}$ for some $j\in\{0,1,\ldots,n-1\}$, and say it is \emph{admissible} if it is non-alternating and satisfies
\[ n + | \{ i: 0\le i\le n,\ d_i=0 \}| \equiv 1 \pmod 3. \]
Write $(a_0,a_1,\ldots,a_n)\equiv (b_0,b_1,\ldots,b_n)\pmod m$ if $a_i\equiv b_i\pmod m$ for all $i\in\{0,1,\ldots,n\}$.

\begin{lemma}\label{lem:admissible}
Let $n\ge0$ and suppose $(d_0,d_1, \ldots,d_n) \in \{0,1\}^{n+1}$ is admissible.
Then there exists a binary tree $t\in\T_n$ such that $d(t) \equiv (d_0,d_1,\ldots,d_n) \pmod 2$.
\end{lemma}

\begin{proof} 
We induct on $n$. 
The result holds for $n=0$ since the only admissible sequence in $\{0,1\}^1$ is $(0)$, which is the depth sequence of the unique binary tree in $\T_0$.

For $n\ge1$, suppose $(d_0,d_1,\ldots,d_n) \in \{0,1\}^{n+1}$ is admissible, i.e., it is non-alternating and satisfies $n+r \equiv 1 \pmod 3$, where $r := | \{ i: 0\le i\le n,\ d_i=0 \}|$.
This sequence must contain a run of all zeros or all ones of length at least two.
Replacing an adjacent pair of zeros (or ones, resp.) in this run with a one (or zero, resp.) gives a new sequence $(d'_0,\ldots,d'_{n-1})\in\{0,1\}^{n}$, which satisfies
\[ n-1 + | \{ i: 0\le i\le n,\ d'_i=0 \}| \equiv n-1+r+1 \equiv 1\pmod 3.\]
We show that it is always possible to make the new sequence $(d'_0,\ldots,d'_{n-1})$ non-alternating.

In fact, we may assume $(d_0,\ldots,d_n)$ contains a run of zeros of length at least two, without loss of generality.
If this run of zeros is preceded (or followed, resp.) by a one, then we can replace the first (or last, resp.) two zeros in this run with a one and get a non-alternating sequence $(d'_0,\ldots,d'_{n-1})$.
If it is not preceded nor followed by a one, then we must have $d_0=d_1=\cdots=d_n=0$.
The condition $n+r\equiv1\pmod 3$ implies $n\equiv0\pmod3$ and thus $n\ge3$.
Then replacing the first two zeros in $(d_0,\ldots,d_n)$ with a one gives a non-alternating sequence $(d'_0,\ldots,d'_{n-1})$.
Some examples are given below.
\[ 1\underline{00}1100 \to 1\underline{1}1100 \qquad
1010\underline{00}1 \to 1010\underline{1}1 \qquad
 0000 \to \underline{1}00 \]
 
Now we know that $(d'_0,\ldots,d'_{n-1})$ is admissible. 
By the induction hypothesis, there exists a binary tree $t'\in\T_{n-1}$ such that $d(t')\equiv (d'_0,\ldots,d'_{n-1}) \pmod 2$.
The construction of the sequence $(d'_0,\ldots,d'_{n-1})$ implies that $d'_i$ is obtained by switching a pair of adjacent zeros (or ones, resp.) in $(d_0,\ldots,d_n)$ to a one (or zero, resp.) for some $i\in\{0,1,\ldots,n-1\}$.
Then appending two children to the $i$th leaf of $t'$ gives a binary tree $t\in\T_n$ satisfying $d(t)\equiv (d_0,\ldots,d_n) \pmod 2$. 
\end{proof}

\section{Main results}\label{sec:main}

In this section we prove that the number $C_{\ominus,n}$ of equivalence classes of parenthesizations of $x_0\ominus x_1\ominus\cdots\ominus x_n$ equals $A_n$, the $n$th term of the sequence \seqnum{A000975}.

Since $C_{\ominus,n}$ also enumerates distinct results from parenthesizing  $x_0\ominus x_1\ominus\cdots\ominus x_n$, it is refined by the number $C_{\ominus,n,r}$ of distinct results from parenthesizations of $x_0\ominus x_1\ominus \cdots \ominus x_n$ with exactly $r$ plus signs.
We have 
\begin{equation}\label{eq:refine}
C_{\ominus,n} = \sum_{0\le r\le n+1} C_{\ominus,n,r}.
\end{equation}
We can establish a closed formula for the number $C_{\ominus,n,r}$, thanks to the following observation.

\begin{proposition}\label{prop:sign}
A parenthesization of $x_0 \ominus x_1\ominus \cdots \ominus x_n$ equals 
\[ (-1)^{d_0(t)} x_0 + (-1)^{d_1(t)} x_1 + \cdots + (-1)^{d_n(t)} x_n \]
where $t\in\T_n$ is the binary tree corresponding to this parenthesization.
\end{proposition}

\begin{proof}
If a parenthesization of $x_0 \ominus x_1\ominus \cdots \ominus x_n$ corresponds to a binary tree $t\in\T_n$, then expanding this parenthesization using the definition $a\ominus b:=-a-b$, one sees that $x_i$ receives $d_i(t)$ many negative signs for all $i$.
The result follows.
\end{proof}

By Proposition~\ref{prop:sign}, $C_{\ominus,n,r}$ enumerates binary trees in $\T_n$ with exactly $r$ leaves of even depth.
Using this observation, we compute the number $C_{\ominus,n,r}$ for $0\le n\le 12$ and $0\le r\le n+1$ and record the result in Table~\ref{tab:Cnr}, whose rows and columns are indexed by $n$ and $r$, respectively.

\begin{table}[H]
\centering
\begin{tabular}{c|cccccccccccccc}
$r$ & $0$ & $1$ & $2$ & $3$ & $4$ & $5$ & $6$ & $7$ & $8$ & $9$ & $10$ & $11$ & $12$ & $13$ \\
\hline
$C_{\ominus,0,r}$ & & 1 & \\
\hline
$C_{\ominus,1,r}$ & 1 & \\
\hline
$C_{\ominus,2,r}$ & & & ${\color{red}2}$ \\
\hline
$C_{\ominus,3,r}$ & & 4 & & & 1 \\
\hline
$C_{\ominus,4,r}$ & 1 & & & {\color{red}9} \\
\hline
$C_{\ominus,5,r}$ & & & 15 & & & 6 \\
\hline
$C_{\ominus,6,r}$ & & 7 & & & {\color{red}34} & & & 1 \\
\hline
$C_{\ominus,7,r}$ & 1 & & & 56 & & & 28 \\
\hline
$C_{\ominus,8,r}$ & & & 36 & & & {\color{red}125} & & & 9 \\
\hline
$C_{\ominus,9,r}$ & & 10 & & & 210 & & & 120 & & & 1 \\
\hline
$C_{\ominus,10,r}$ & 1 & & & 165 & & & {\color{red}461} & & & 55 \\
\hline
$C_{\ominus,11,r}$ & & & 66 & & & 792 & & & 495 & & & 12 \\
\hline
$C_{\ominus,12,r}$ & & 13 & & & 715 & & & {\color{red}1715} & & & 286 & & & 1 \\
\end{tabular}
\caption{$C_{\ominus,n,r}$ for $0\le n\le 12$ and $0\le r\le n+1$}\label{tab:Cnr}
\end{table}

In this table, an empty spot means zero, a black entry equals $\binom{n+1}{r}$, and a red entry equals $\binom{n+1}{r}-1$.
Table~\ref{tab:Cnr} can be viewed as a truncated/modified version of the well-known Pascal triangle and suggests a closed formula for $C_{\ominus,n,r}$.

\begin{theorem}\label{thm:Cnr}
For $n\ge1$ and $0\le r\le n+1$ we have
\[ C_{\ominus,n,r} =
\begin{cases}
\displaystyle \binom{n+1}{r} , & \text{if } n+r \equiv 1 \pmod 3 \text{ and } n \ne 2r-2;  \\[12pt]
\displaystyle \binom{n+1}{r} - 1, & \text{if } n+r \equiv 1 \pmod 3 \text{ and } n = 2r-2; \\[12pt]
0, & \text{if } n+r \not\equiv 1 \pmod 3.
\end{cases} \]
\end{theorem}

\begin{proof}
Lemma~\ref{lem:DNE} gives $C_{\ominus,n,r}=0$ if $n+r\not\equiv1\pmod3$.
Assume $n+r\equiv1\pmod 3$ below.

By definition, $C_{\ominus,n,r}\le \binom{n+1}{r}$.
Lemma~\ref{lem:alternating} and \ref{lem:admissible} imply that $\binom{n+1}{r} -C_{\ominus,n,r}$ equals the number of sequences $\mathbf{d} := (d_0,d_1,\ldots,d_n)\in\{0,1\}^{n+1}$ such that $\{i: 0\le i\le n,\ d_i=0\}|=r$ and $d_j\ne d_{j+1}$ for all $j\in\{0,1,\ldots,n-1\}$.
Such a sequence $\mathbf{d}$ must be alternating and we distinguish the following cases.

\vskip3pt\noindent\textsf{Case 1.}
If $\mathbf{d} = (1,0,1,0,\ldots,0,1)$ then $n=2r$ and $n+r=3r\not\equiv1\pmod 3$.

\vskip3pt\noindent\textsf{Case 2.}
If $\mathbf{d} = (1,0,1,0,\ldots,1,0)$ or $\mathbf{d} = (0,1,0,1,\ldots,0,1)$ then we have $n=2r-1$ and $n+r\equiv 3r-1\not\equiv 1\pmod 3$.

\vskip3pt\noindent\textsf{Case 3.}
If $\mathbf{d} = (0,1,0,1,\ldots,1,0)$ then $n=2r-2$ and $n+r\equiv 3r-2\equiv 1\pmod 3$.

Thus 
\[ \binom{n+1}{r} -C_{\ominus,n,r} = 
\begin{cases}
0, & n\ne 2r-2; \\
1, & n = 2r-2.
\end{cases}\]
This establishes the desired formula for $C_{\ominus,n,r}$.
\end{proof}

Now we are ready to give the main result on the number $C_{\ominus,n}$.

\begin{theorem}\label{thm:ominus}
For all $n\ge1$ we have $C_{\ominus,n} = A_n$.
\end{theorem}

\begin{proof}
Let $\omega=e^{2\pi i/3}$ be a primitive cubic root of unity in $\CC$. 
Substituting $x=1, \omega, \omega^2$ in the binomial expansion of $(1+x)^n$ 
we obtain
\[ 2^n+\omega^{-k}(1+\omega)^n+\omega^k(1+\omega^2)^n = \sum_i \binom{n}{i}(1+\omega^{i-k} + \omega^{2i+k} ). \] 
We have
\[ 1+\omega^{i-k} + \omega^{2i+k} = 
\begin{cases}
1+1+1=3, & \text{ if $i\equiv k\pmod 3$}, \\
1+\omega+\omega^2 =0, & \text{ otherwise}.
\end{cases} \]
Hence
\begin{equation}\label{eq:skipping}
 \sum_{i\equiv k \!\!\!\! \pmod 3} \binom{n}{i} = \frac{2^n+\omega^{-k}(1+\omega)^n+\omega^k(1+\omega^2)^n}3. 
\end{equation}

For $n\ge1$, let 
\[ C'_{\ominus,n} :=
\begin{cases}
C_{\ominus,n}, & \text{if $n$ is odd}; \\
C_{\ominus,n}+1, & \text{if $n$ is even}.
\end{cases}\]
It follows from Equation~\eqref{eq:refine}, Theorem~\ref{thm:Cnr}, and Equation~\eqref{eq:skipping} that
\begin{align*}
C'_{\ominus,n} &= \sum_{r\equiv 1-n \!\!\!\! \pmod 3} \binom{n+1}{r} \\
&= \frac{2^{n+1}+\omega^{n-1}(1+\omega)^{n+1}+\omega^{1-n}(1+\omega^2)^{n+1}}3 \\
&= \frac{2^{n+1}+\omega^{n-1}(-\omega^2)^{n+1} + \omega^{1-n}(-\omega)^{n+1} } 3 \\
&= \frac{2^{n+1}+(-1)^{n+1} \omega + (-1)^{n+1} \omega^{2} } 3 
= \frac{ 2^{n+1} + (-1)^n} 3.\\
\end{align*}
This implies that $C_{\ominus,n}$ satisfies the same closed formulas~\eqref{eq:An} as $A_n$.
\end{proof}

We are grateful to the anonymous referee for pointing out the following simpler proof of Theorem~\ref{thm:ominus}.

\begin{proof}[Another Proof]
We have $C_{\ominus,1}=A_1=1$.
For $n\ge2$, applying Theorem~\ref{thm:Cnr} we obtain
\begin{align*}
C_{\ominus,n} + C_{\ominus,n-1} + 1 & = 
\sum_{ r\equiv 1-n\!\!\!\! \pmod 3 } \binom{n+1}{r} 
+ \sum_{r \equiv 2-n\!\!\!\! \pmod 3 } \binom{n}{r}\\
& = \sum_{ r\equiv 1-n \!\!\!\! \pmod 3 } \left( \binom{n}{r-1} + \binom{n}{r} \right) 
+ \sum_{ r\equiv 2-n \!\!\!\! \pmod 3 } \binom{n}{r}\\
& = \sum_{0\le r \le n} \binom{n}{r} = 2^n.
\end{align*}
Thus $C_{\ominus,n}$ satisfies the same recurrence relation as $A_n=2^n-1-A_{n-1}$ for $n\ge2$.
\end{proof}

Finally, we characterize all distinct results from parenthesizing $x_0\ominus x_1\ominus\cdots\ominus x_n$.

\begin{proposition}\label{prop:result}
The results from parenthesizing $x_0\ominus x_1\ominus \cdots \ominus x_n$ are precisely
\[ (-1)^{d_0} x_0 + (-1)^{d_1} x_1 + \cdots + (-1)^{d_n} x_n \] 
for all admissible sequences $(d_0,d_1,\ldots,d_n)\in\{0,1\}^{n+1}$.
Consequently, sequence \seqnum{A000975} enumerates admissible binary sequences.
\end{proposition}

\begin{proof}
Any result from parenthesizing $x_0\ominus x_1\ominus \cdots \ominus x_n$ can be written as
\[ (-1)^{d_0(t)} x_0 + (-1)^{d_1(t)} x_1 + \cdots + (-1)^{d_n(t)} x_n \] 
where $t\in\T_n$ corresponds to this parenthesization.
Let $(d_0,\ldots,d_n)\in\{0,1\}^{n+1}$ such that $(d_0,\ldots,d_n)\equiv(d_0(t),\ldots,d_n(t)) \pmod2$. 
Then $(d_0,\ldots,d_n)$ must be admissible by Lemma~\ref{lem:DNE} and \ref{lem:alternating}.

Conversely, if $(d_0,\ldots,d_n)\in\{0,1\}^{n+1}$ is admissible then there exists $t\in\T_n$ such that $d(t)\equiv(d_0,\ldots,d_n)\pmod2$ by Lemma~\ref{lem:admissible}.
Thus the parenthesization of $x_0\ominus x_1\ominus \cdots \ominus x_n$ corresponding $t$ equals
$(-1)^{d_0} x_0 + (-1)^{d_1} x_1 + \cdots + (-1)^{d_n} x_n$.
\end{proof}

\section{Questions}\label{sec:questions}

In this section we present some problems for future study.

First, similarly to many other enumeration problems, it would be nice to have explicit bijections between distinct results from parenthesizations of $x_0\ominus x_1\ominus \cdots \ominus x_n$, which are naturally encoded by admissible sequences in $\{0,1\}^{n+1}$, and other objects enumerated by sequence \seqnum{A000975}, such as those mentioned in Section~\ref{sec:intro}.
Such bijections may help find a direct proof for the closed formulas of $C_{\ominus,n}$ without using the refined number $C_{\ominus,n,r}$. 

Next, $C_{\ominus,n}$ satisfies the recurrent relations for $A_n$ mentioned in Section 1 since $C_{\ominus,n}=A_n$. 
The second proof of Theorem~\ref{thm:ominus} establishes one of these recurrence relations directly for $C_{\ominus,n}$. 
How about the other recurrent relations? 

It also seems to natural to explore the generating function 
\[ C_{\ominus}(x,y) := \sum_{n\ge0} \sum_{0\le r\le n+1} C_{\ominus,n,r} x^ny^r. \]
Its specializations $C_\ominus(x,1)$ is known to be
\[ C_\ominus(x,1) = \sum_{n\ge0} C_{\ominus,n} x^n 
= \frac{1}{(1+x)(1-x)(1-2x)}\] 
which gives the generating function for the sequence \seqnum{A000975} (although $A_0=0$ differs from $C_{\ominus,0}=1$).

Lastly, the results from parenthesizations of $x_0\ominus x_1\ominus\cdots\ominus x_n$ are determined by the depths of leaves in the corresponding binary trees.
What is the average leaf depth in all binary trees with $n+1$ leaves?

\section{Acknowledgments}
We are grateful to the anonymous referee for providing many valuable suggestions, which helped us improve the organization of this paper and simplified some proofs.

\bigskip
\hrule
\bigskip

\noindent 2010 {\it Mathematics Subject Classification}: Primary 05A15.

\noindent \emph{Keywords:}
Binary tree, leaf depth, binary sequence, nonassociativity, parenthesization.

\bigskip
\hrule
\bigskip

\noindent (Concerned with sequences
\seqnum{A000217}, 
\seqnum{A000975},
\seqnum{A048702}, 
 \seqnum{A155051}, and
\seqnum{A265158}.)


\begin{thebibliography}{50}

\bibitem{CatMod}
N. Hein and J. Huang, Modular Catalan numbers, \emph{Euro. J.  Comb.} \textbf{61} (2017), 197--218.

\bibitem{Hinz}
A. M. Hinz, The Lichtenberg sequence, \emph{Fibonacci Quart.} \textbf{55} (2017), 2--12.

\bibitem{HanoiTower}
A. M. Hinz, S. Klav\v{z}ar, U. Milutinovi\'{c}, and C. Petr, \emph{The Tower of Hanoi---Myths
and Maths}, Springer, Basel, 2013.

\bibitem{CircSet}
D. E. Knuth and O. P. Lossers, Partitions of a circular set: 11151, \emph{Amer. Math. Monthly} \textbf{114} (2007), 265--266.

\bibitem{Lord}
N. J. Lord, Non-associative operations, \textit{Math. Mag.} \textbf{60} (1987), 174--177. 

\bibitem{OEIS}
N. J. Sloane, ed., \textit{The On-Line Encyclopedia of Integer Sequences}, \url{http://oeis.org}, 2017.

\bibitem{EC2}
R. Stanley, \textit{Enumerative Combinatorics, Volume 2}, Cambridge University Press, 1999.

\bibitem{A000975}
P. K. Stockmeyer, An exploration of sequence A000975, \textit{Fibonacci Quart.} \textbf{55} (2017), 174--185.

\end{thebibliography}
\end{document}